  \documentclass[reqno, 11 pt]{amsart}
  
\usepackage{mathtools}
\usepackage{amssymb}   
\usepackage{amsmath}   
\usepackage{euscript}   
\usepackage{hyperref} 
\usepackage{amsthm} 
\usepackage{cleveref} 
\usepackage{dsfont}  
\usepackage{amsfonts}    
\usepackage{marginnote}
\usepackage{latexsym} 
\usepackage{amsopn} 
\usepackage{geometry}
\usepackage{amscd} 
\usepackage{mathrsfs}
\usepackage{caption}
\usepackage{aurical}  
\usepackage[english]{babel}
\usepackage[utf8]{inputenc}
\usepackage{esint} 
\usepackage{graphicx}
\usepackage[dvipsnames]{xcolor}
\usepackage{multicol}
\usepackage{nomencl}
\makeglossary
\makenomenclature  
\usepackage{refcount}
\usepackage{float}
\usepackage{times}
\usepackage[titletoc]{appendix}

\newcommand{\comment}[1]{}

\newcommand{\pa}{\partial}

\definecolor{airforce}{rgb}{0.36, 0.74, 0.86}
 	\definecolor{blue-green}{rgb}{0.0, 0.57, 0.87}

\newcommand{\s}{{\sigma}}

\renewcommand{\Im}{\operatorname{Im}}

\newcommand{\Id}{\mathds{1}}

\newcommand{\opbw}{{Op^{\mathrm{BW}}}}

\newcommand{\gr}[1]{\textbf{#1}}

\providecommand{\vect}[2]{{\bigl[\begin{smallmatrix}#1\\#2\end{smallmatrix}\bigr]}}   
\providecommand{\sm}[4]{{\bigl[\begin{smallmatrix}#1&#2\\#3&#4\end{smallmatrix}\bigr]}}

\newtheorem{theorem}{Theorem}[section]
\newtheorem*{thm*}{Theorem}
\newtheorem{proposition}[theorem]{Proposition}
\newtheorem{lemma}[theorem]{Lemma}

\newtheorem*{cor*}{Corollary}
\newtheorem{remark}[theorem]{Remark}
\newtheorem{ex}{Example}

\newtheorem{definition}[theorem]{Definition}

\numberwithin{equation}{section}

\newcommand{\ii}{{\rm i}}

\def\U{\underline{{\bf U}}}

\newcommand{\x}{\xi}

\newcommand{\ov}{\overline}

\newcommand{\C}{{\mathbb C}}

\newcommand{\N}{{\mathbb N}}

\newcommand{\R}{{\mathbb R}}

\newcommand{\T}{{\mathbb T}}
\newcommand{\Z}{{\mathbb Z}}

\newcommand{\cA}{{\mathcal A}}

\newcommand{\cC}{{\mathcal C}}

\newcommand{\cH}{{\mathcal H}}

\newcommand{\cP}{{\mathcal P}}

\usepackage{bm}

\newcommand{\nnorm}[1]{{\left\vert\kern-0.25ex\left\vert\kern-0.25ex\left\vert #1 
    \right\vert\kern-0.25ex\right\vert\kern-0.25ex\right\vert}}

\usepackage{framed,enumitem}

\providecommand{\vect}[2]{{\bigl[\begin{smallmatrix}#1\\#2\end{smallmatrix}\bigr]}} 
  
\providecommand{\sm}[4]{{\bigl[\begin{smallmatrix}#1&#2\\#3&#4\end{smallmatrix}\bigr]}}

\makeatletter

\renewcommand{\tocsection}[3]{%
\indentlabel{\@ifnotempty{#2}{\bfseries\ignorespaces#1 #2\quad}}\bfseries#3}
\def\l@subsection{\@tocline{2}{0pt}{2.5pc}{5pc}{}}
\def\l@subsubsection{\@tocline{3}{0pt}{4.5pc}{5pc}{}}
\renewcommand\tocchapter[3]{%
  \indentlabel{\@ifnotempty{#2}{\ignorespaces#2.\quad}}#3%
}

\begin{document} 
 
\title
{On the quasilinear Schr\"odinger equations on tori.}
\date{}

\author{Felice  Iandoli}
\address{\scriptsize{Dipartimento di Matematica ed Informatica, 
Universit\`a della Calabria, Ponte Pietro Bucci,  87036, Rende, Italy}}
\email{felice.iandoli@unical.it}



   
\begin{abstract}    
We improve the result by Feola and Iandoli [J. de Math. Pures et App., 157:243--281, 2022], showing that quasilinear Hamiltonian Schr\"odinger type equations are well posed on $H^s(\T^d)$ if $s>d/2+3$. We exploit the sharp paradifferential calculus on $\T^d$ developed by Berti, Maspero and Murgante [J. Dynam. Differential Equations, 33(3):1475--1513, 2021].
\end{abstract}  
    
\maketitle

\setcounter{tocdepth}{1}
\tableofcontents


\section{Introduction}
In this paper we  study the local in time solvability of the Cauchy problem associated to the following quasilinear  Schr\"odinger equation
\begin{equation}\label{NLS}
\ii u_{t}+(\pa_{\ov{u}}F)(u,\nabla u)-\sum_{j=1}^{d}\pa_{x_j}\big(\pa_{\ov{u}_{x_{j}}}F\big)(u,\nabla u)=0 \,,
\end{equation}
where $
u=u(t,x)\,$ and $ 
 x=(x_1,\ldots,x_{d})\in \mathbb{T}^{d}:=(\mathbb{R}/2\pi \mathbb{Z})^{d}
 $
and 
where we denoted $\pa_{u}:=(\pa_{{\rm Re}(u)}-\ii \pa_{{\rm Im}(u)})/2$ and  
$\pa_{\bar{u}}:=(\pa_{{\rm Re}(u)}+\ii \pa_{{\rm Im}(u)})/2$
the Wirtinger derivatives. The function
$F(y_0,y_1,\ldots,y_{d})$  is a \emph{real} valued polynomial in the complex variables $(y_0,\ldots,y_d)\in \C^{d+1}$. 
   Here $
 \nabla u=(\pa_{x_{1}} u, \ldots, \pa_{x_d}u)
 $ is the gradient. 
 Notice that equation \eqref{NLS} is \emph{Hamiltonian}, i.e.
 \begin{equation}\label{HamNLS}
 u_{t}=\ii \nabla_{\bar{u}}H(u,\bar{u})\,,
\quad 
H(u,\bar{u}):=\int_{\mathbb{T}^{d}}F(u,\nabla u) dx\,,
\end{equation}
where $\nabla_{\bar{u}}:=(\nabla_{{\rm Re}(u)}-\ii\nabla_{{\rm Im}(u)})/2$
and $\nabla$ denote the $L^{2}$-gradient. 
We assume  that the function $F$, defining the non-linearity, satisfies 
following \emph{ellipticity condition}:
there exists  $\mathtt{c}>0$
such that 
for any $\x=(\x_1,\ldots,\x_{d})\in \mathbb{R}^{d}$, 
$y=(y_0,\ldots,y_d)\in \mathbb{C}^{d+1}$
one has
\begin{equation}\label{elli}
\sum_{j,k=1}^{d}\x_{j}\x_{k}\pa_{y_j}
\pa_{\ov{y}_{k}}F(y)
-\Big|\sum_{j,k=1}^{d}\x_{j}\x_{k}\pa_{\ov{y}_j}\pa_{\ov{y}_k}F(y) 
\Big|
\geq \mathtt{c}|\xi|^2\,.
\end{equation}
Note that this condition  implies that 
\begin{align*}
\sum_{j,k=1}^{d}\x_{j}\x_{k}\pa_{y_j}
\pa_{\ov{y}_{k}}F(y)&\geq \mathtt{c}|\xi|^2,\\
\Big(\sum_{j,k=1}^{d}\x_{j}\x_{k}\pa_{y_j}
\pa_{\ov{y}_{k}}F(y)\Big)^2
-\Big|\sum_{j,k=1}^{d}\x_{j}\x_{k}\pa_{\ov{y}_j}\pa_{\ov{y}_{k}}F(y) 
\Big|^2
&\geq \mathtt{c}|\xi|^2,
\end{align*}
which are the assumptions made in the paper \cite{FIJMPA}.
The main result of this paper is the following local well posedness theorem on Sobolev spaces (see Section \ref{para}).
\begin{theorem}{\bf (Local well-posedness).}\label{main}
Let $s\geq \frak{s}_0> d/2+3$,
consider the equation \eqref{NLS}
with initial condition $u(0,x)=u_0(x)$ in $H^s(\mathbb{T}^d;\mathbb{C})$,  and with $F$ satisfying \eqref{elli}.
There exists a time 
$0<T=T(\|u_0\|_{\frak{s}_0})$ 
and a unique solution 
$
u(t,x)$ of \eqref{NLS} in $C^0([0,T),H^s)
\cap C^1([0,T),H^{s-2}).$  
Moreover the solution map $u_0(x)\mapsto u(t,x)$ is continuous with respect to the $H^s$ topology for any $t$ in $[0,T)$.
\end{theorem}
Theorem \ref{main} improves the results in \cite{FIJMPA, FIHP} in terms of the required regularity on the initial datum. Furthermore the time of existence of the solutions depends only on the low norm $\|\cdot\|_{\frak{s}_0}$. The regularity threshold $d/2+3$ is the same required by Marzuola, Metcalfe and Tataru in \cite{MMT3} (which improves the pioneering result by Kenig, Ponce and Vega \cite{KPV}) for general quasilinear Schr\"odinger equations on $\R^d$. Moreover, as in \cite{MMT3}, when the function defining the Hamiltonian has the special form $F(u,\nabla u)=|\nabla h(|u|^2)|^2$ we can improve the result requiring only $s>d/2+2$. We have the following.
\begin{theorem}\label{vjj}
If, in addition to the hypotheses of Theorem \ref{main}, one has $F(u,\nabla u)=|\nabla h(|u|^2)|^2$ for some polynomial function $h:\C\rightarrow \C$, then the result of Theorem \ref{main} holds true for any $\frak{s}_0>d/2+2$.
\end{theorem}
The equation satisfying the hypotheses of Theorem \ref{vjj} is the one considered in \cite{IN2023,FGI, colin}.\\
A very efficient way to prove the well posedness for a quasilinear equation, is to paralinearize it and consider a quasilinear iterative scheme \`a la Kato \cite{Kato} to build up the solutions. The key step, in order to pass to the limit in the iterative scheme,
 is to obtain some \emph{a priori} estimates on the solutions of the paralinearized system. This method has been refined recently, see for instance \cite{FIHP, FIJMPA, BMM1, MMT3}.\\
In \cite{FIJMPA} the authors introduce a non sharp paradifferential calculus on $\T^d$, which has been refined by Berti, Maspero and Murgante in \cite{BMM1}.
Moreover in \cite{FIJMPA}, to obtain \emph{a priori} estimates, the authors perform some changes of coordinates of $H^s$ in order to diagonalize the paralinearized  system at the positive orders. In \cite{BMM1}, instead of changing unknown,  the authors, inspired also by Alazard, Burq and Zuily \cite{ABZ}, introduce a modified norm in $H^s$ which is tailored to the problem and which allows them to obtain the wanted estimates. This is less expensive in terms of amount of regularity, for this reason  we follow closely this approach. The main difference with respect to \cite{BMM1} is that the the matrix of symbols which diagonalizes the principal order, see \eqref{transC}, depends on $\xi$, $u$ and  $\nabla u$, while in \cite{BMM1} (and in the case that $F$ satisfies the hypotheses of Theorem \ref{vjj}) it depends only on  $u$. The $\xi$-dependence  imposes a lower order correction in the diagonalization process, see \eqref{simboliAA11}, in order to have a well defined parametrix with a remainder gaining two derivatives. 
In \cite{BMM1},  to show the existence for the paralinearized problem, the authors use a finite rank projection with a cut-off which is tailored to the problem. It is not clear to the author of the present paper if this approach works for lower order symbolic calculus, which is needed here for the aforementioned reasons  for the general equation \eqref{NLS} (it clearly works in the case that the nonlinearity satisfies the hypothesis of Theorem \ref{vjj}). We use, instead, the artificial viscosity approximation to build the solutions. A Garding type inequality is therefore needed, see Lemma \ref{garding}, to close the energy estimates.

The \eqref{NLS} is slightly more general  than the equation considered in \cite{FIJMPA}, where the presence of the linear term $\Delta u$ was required. The presence of this linear term was somehow used in the proof of the existence of the solutions of the paralinearized problem.

We conclude this introduction quoting the recent paper by Jeong and Oh \cite{JO2023} about the ill posedness in $\R^d$ of the problem in the case that \eqref{elli} is not satisfied.
We also mention the paper by Christ \cite{Christ}, in which the author gives some examples of non Hamiltonian Schr\"odinger type equations which are ill posed on the circle.\\

\section{Paradifferential calculus}\label{para}
We fix some notation concerning Sobolev spaces.
We denote by $H^{s}(\mathbb{T}^d;\mathbb{C})$
(resp. $H^{s}(\mathbb{T}^d;\mathbb{C}^{2})$)
the usual Sobolev space of functions $\mathbb{T}^{d}\ni x \mapsto u(x)\in \mathbb{C}$
(resp. $\C^{2}$).
We expand a function $ u(x) $, $x\in \mathbb{T}^{d}$, 
 in Fourier series as 
\begin{equation}\label{complex-uU}
u(x) = {(2\pi)^{-{d}/{2}}}
\sum_{n \in \mathbb{Z}^{d} } \hat{u}(n)e^{\ii n\cdot x } \, , \qquad 
\hat{u}(n) := {(2\pi)^{-{d}/{2}}} \int_{\mathbb{T}^{d}} u(x) e^{-\ii n \cdot x } \, dx \, .
\end{equation}
We also use the notation
$u_n := \hat{u}(n)$ and $
  \ov{u_n}  := \ov{\hat{u}(n)} \,. $ 
We set $\langle j \rangle:=\sqrt{1+|j|^{2}}$ for $j\in \mathbb{Z}^{d}$.
We endow $H^{s}(\mathbb{T}^{d};\mathbb{C})$ with the norm 
\begin{equation}\label{Sobnorm}
\|u(\cdot)\|_{s}^{2}:=\sum_{j\in \mathbb{Z}^{d}}\langle j\rangle^{2s}|u_{j}|^{2}\,.
\end{equation}
For $U=(u_1,u_2)\in H^{s}(\mathbb{T}^d;\mathbb{C}^{2})$
we just set
$\|U\|_{{s}}=\|u_1\|_{{s}}+\|u_2\|_{{s}}$.
We shall also write the norm in \eqref{Sobnorm} as
\begin{equation*}
\|u\|^{2}_{{s}}=(\langle D\rangle^{s}u,\langle D\rangle^{s} u)_{L^{2}}\,, 
\qquad
\langle D\rangle e^{\ii j\cdot x}=\langle j\rangle  e^{\ii j\cdot x}\,,\;\;\; 
 \mbox{for all}\,\, \, j\in \mathbb{Z}^{d}\,,
\end{equation*}
where $(\cdot,\cdot)_{L^{2}}$ denotes the standard complex $L^{2}$-scalar product
 \begin{equation}\label{scalarproScalare}
 (u,v)_{L^{2}}:=\int_{\mathbb{T}^{d}}u\bar{v}dx\,, 
 \qquad  \mbox{for all}\,\,\, u,v\in L^{2}(\mathbb{T}^{d};\mathbb{C})\,.
 \end{equation}
  \noindent
  We introduce also the product  spaces 
\begin{equation}\label{RealSobolev}
\mathcal{H}^s=\{ U=(u^+, u^-)\in 
H^s\times H^s : u^+=\overline{u^-} \}\,.
\end{equation}
With abuse of notation we shall denote by $\|\cdot\|_{{s}}$ the 
natural norm of the product space $\mathcal{H}^{s}$.
On the space $ \mathcal{H}^s$
we naturally extend the scalar product \eqref{scalarproScalare} as
\begin{equation}\label{scalarprod2x2}
(Z,W)_{L^{2}\times L^{2}}:={\rm Re}(z,w)_{L^{2}}=\tfrac{1}{4\pi}\int_{\T} z\cdot \bar{w}+\bar{z}\cdot w dx\,,
\qquad Z=\vect{z}{\bar{z}}\,,\; W=\vect{w}{\bar{w}}\in \mathcal{H}^{0}\,.
\end{equation}

  We recall the following interpolation estimate
\begin{equation}\label{interpolo}
\|u\|_{{s}}\leq \|u\|^{\theta}_{{s_1}}\|u\|^{1-\theta}_{{s_2}}\,,
\qquad \theta\in[0,1]\,,\;\;0\leq s_1\leq s_2\,,\quad s=\theta s_1+(1-\theta)s_2\,.
\end{equation}
{\bf Notation}. 
We shall 
use the notation $A\lesssim B$ to denote 
$A\le C B$ where $C$ is a positive constant
depending on  parameters fixed once for all, for instance $d$
 and $s$.
 We will emphasize by writing $\lesssim_{q}$
 when the constant $C$ depends on some other parameter $q$. When we have both $A\lesssim B$ and $B\lesssim A$ we shall write $A\sim B$.

We now recall some results concerning the paradifferential calculus. 
We  follow \cite{BMM1}.

\begin{definition}\label{def:simbolini}
Given $m,s\in\mathbb{R}$ we denote by $\Gamma^m_s$ 
the space of functions $a(x,\xi)$ defined on $\T^d\times \R^d$ 
with values in $\C$, which are $C^{\infty}$ 
with respect to the variable $\xi\in\R^d$ and such that for any $\beta\in \N\cup\{0\}$,
there exists a constant $C_{\beta}>0$ such that 
\begin{equation}\label{stima-simbolo}
\|\partial_{\xi}^{\beta} a(\cdot,\xi)\|_{s}
\leq C_{\beta}\langle\xi\rangle^{m-|\beta|}\,, 
\quad  \mbox{for all}\,\, \xi\in\R^d.
\end{equation}
The elements of $\Gamma_s^m$ are called \emph{symbols} of order $m$.
\end{definition}
We endow the space $\Gamma^m_{s}$ with the family of {semi}-norms 
\begin{equation}\label{seminorme}
|a|_{m,s,n}:=\max_{\beta\leq n}\, 
\sup_{\xi\in\R^d}\|\langle\xi\rangle^{\beta-m}\partial^\beta_\xi a(\cdot,\xi)\|_{s}\,.
\end{equation}
Consider a function $\chi\in C^{\infty}(\R^d,[0,1])$ such that 
\[
\chi(\xi)=\begin{cases}
1 \qquad \mathrm{if}\,\,|\xi|\leq 1.1,\\
0 \qquad \mathrm{if}\,\,|\xi|\geq 1.9.
\end{cases}
\]
Let $\epsilon\in(0,1)$ and define 
$\chi_{\epsilon}(\xi):=\chi(\xi/\epsilon).$ 
Given a symbol $a(x,\x)$ in $\Gamma^m_s$
we define its Bony-Weyl
quantization 
\begin{equation}\label{quantiWeyl}
T_{a}h:=\opbw(a(x,\xi))h:=\sum_{j\in \mathbb{Z}^d}e^{\ii j x}
\sum_{k\in\mathbb{Z}^d}
\chi_{\epsilon}\left(\tfrac{|j-k|}{\langle j+k\rangle}\right)
\widehat{a}\left(j-k,\tfrac{j+k}{2}\right)\widehat{h}(k).
\end{equation}
We list a series of technical lemmas that we shall systematically use throughout the paper.
The following is Theorem $2.4$ in \cite{BMM1}  and 
concerns the action of paradifferential operators on Sobolev spaces.
\begin{theorem}{\bf (Action on Sobolev spaces).}\label{azione}
Let $s_0>d/2$, $m\in\R$ and $a\in\Gamma^m_{s_0}$. Then $\opbw(a)$ 
extends as a bounded operator from ${H}^{s}(\T^d;\C)$ to ${H}^{s-m}(\T^d;\C)$, 
for all $s\in\R$, with the following estimate  
\begin{equation}\label{ac1}
\|\opbw(a)u\|_{{s-m}}\lesssim |a|_{m,s_0,2(d+1)}\|u\|_{s} \qquad  \mbox{for all}\,\, \,u\in{H}^s\,.
\end{equation}
Moreover for any $\rho\geq 0$ we have 
\begin{equation}\label{ac2}
\|\opbw(a)u\|_{{s-m-\rho}}\lesssim |a|_{m,s_0-\rho,2(d+1)}\|u\|_{s} \qquad \mbox{for all}\,\,\, {u}\in {H}^s\,.
\end{equation}
\end{theorem}
%

\begin{definition}
Let $\rho\in (0, 2]$. Given two symbols $a\in \Gamma^m_{s_0+\rho}$ and $b\in \Gamma^{ m'}_{s_0+\rho}$, we define
\begin{equation}\label{cancelletto}
a\#_{\rho} b= \begin{cases}
ab \quad \qquad \quad \quad\,\, \rho\in(0,1]\\
ab+\frac{1}{2\ii}\{a,b\}\quad \rho\in (1,2],\\
\end{cases}
\end{equation}
where we denoted by $\{a,b\}:=\nabla_{\xi}a\cdot\nabla_xb-\nabla_xa\cdot\nabla_{\xi}b$ the Poisson's bracket between symbols.
\end{definition}

The following is Theorem $2.5$ in \cite{BMM1}. This  result concerns the symbolic calculus for the composition of Bony-Weyl paradifferential operators. 
\begin{theorem}{\bf (Composition).}\label{compo}
Let $s_0>d/2$, $m,m'\in\mathbb{R}$, $\rho\in(0,2]$ and 
$a\in\Gamma^m_{s_0+\rho}$, $b\in\Gamma^{m'}_{s_0+\rho}$. We have 
\[
\opbw(a)\circ\opbw(b)=\opbw(a\#_{\rho}b)+R^{c}_{\rho}(a,b)\,,
\]
 where the linear operator $R^{c}_{\rho}$ maps ${H}^s(\T)$ into ${H}^{s+\rho-m-m'}$, 
 for any $s\in\R$.
 
 Moreover it satisfies the following estimate, for all $u\in {H}^s$,
\begin{equation}\label{resto}
\|R^{c}_{\rho}(a,b)u\|_{{s-(m+m')+\rho}}
\lesssim_s 
(|a|_{m,s_0+\rho,N}|b|_{m',s_0,N}+|a|_{m,s_0,N}|b|_{m',s_0+\rho,N})\|u\|_{s},
\end{equation}
where $N\geq 3d+4$.\end{theorem}

The next result is Lemma $2.7$ in \cite{BMM1}.
\begin{lemma}{\bf (Paraproduct).} \label{ParaMax} 
Fix $s_0>d/2$ and let $f\in H^s$ and $g\in H^r$ with $s + r \ge 0$. Then
\begin{equation}
\label{eq: paraproductMax}
fg=\opbw(f)g+\opbw({g})f+{R}^p(f,g)\,,
\end{equation}
where the bilinear operator $R: H^s \times H^r\to H^{s + r - s_0}$ satisfies the estimate
\begin{equation}
\label{eq: paraproductMax2}
\|{R}^p(f,g)\|_{{s+r - s_0}} \lesssim_s \|f\|_{s} \|g\|_{r}\,.
\end{equation}
\end{lemma}

\section{Paralinearization} We rewrite equation \eqref{NLS} as a system of paradifferential equations as done in \cite{FIHP,FIJMPA}, but we use the paradifferential framework introduced above. 
Let $u\in H^{{s}}$, with ${s}>d/2+1$, we introduce the symbols
\begin{equation}\label{simboa2}
\begin{aligned}
&a_2(x,\x):=a_{2}(U;x,\x)
:=\sum_{j,k=1}^{d}(\pa_{\ov{u}_{x_k} u_{x_{j}}}F) \x_{j}\x_{k}\,\in \Gamma_{s}^2,
\\&
b_2(x,\x):=b_{2}(U;x,\x)
:=\sum_{j,k=1}^{d}(\pa_{\ov{u}_{x_k}\, \ov{u}_{x_{j}}}F) \x_{j}\x_{k}\,\in\Gamma_{s}^2,\\
&a_1(x,\x):=a_{1}(U;x,\x):=\frac{\ii}{2}\sum_{j=1}^{d}\Big(
(\pa_{\ov{u} u_{x_{j}}}F)-(\pa_{{u} \ov{u}_{x_{j}}}F)
\Big)\x_{j}\,\in \Gamma_{s}^1,
\end{aligned}
\end{equation}
where $F=F(u,\nabla u)$ in \eqref{HamNLS}.
Furthermore, since the symbols may contain one derivative of the unknown $u$, one proves the estimates on the seminorms (recall \eqref{seminorme}) 
\begin{equation}\label{semi-1}
|a_2(x,\x)|_{2,p,\alpha}+ |b_2(x,\x)|_{2,p,\alpha}+ |\vec{a}_1(x)\cdot\xi|_{2,p,\alpha}\lesssim_s \mathtt{C}(\|u\|_{p+1}), \quad \mbox{for all}\,\,\, \tfrac{d}{2}<p.
\end{equation}

 We have the following.
\begin{proposition}\label{NLSparapara}
Fix $s_0>d/2$. The equation \eqref{NLS}
is equivalent to the following system:
\begin{equation}\label{QNLS444}
\dot{U}=\ii E\opbw\big(A_{2}(x,\x)+A_{1}(x,\x)\big)U+
R(U)(U)\,,
\end{equation}
where 
\begin{equation}\label{matrici}
U:=\vect{u}{\bar{u}}\,,\quad  E:=\sm{1}{0}{0}{-1}\,,
\end{equation}
the matrices $A_2(x,\x)=A_2(U;x,\x)$,
$A_1(x,\x)=A_1(U;x,\x)$  have the form
\begin{equation}\label{matriceA2}
A_2(x,\x):=\left(\begin{matrix}a_2(x,\x) & b_{2}(x,\x) \vspace{0.2em}\\
\ov{b_{2}(x,-\x)} & {a_{2}(x,\x)} \end{matrix}\right)\,,
\qquad
A_1(x,\x):=\left(\begin{matrix}a_1(x,\x) & 0 \vspace{0.2em}\\
0 & \ov{a_{1}(x,-\x)} \end{matrix}\right)
\end{equation}
and  $a_2,a_1,b_2$ are  the symbol in \eqref{simboa2}.
For any $s\geq s_0+3$ and any $ U,V\in \cH^{s}(\mathbb{T}^{d};\mathbb{C}^{2}) $, 
the remainder $R$ satisfies the estimates 
\begin{align}
&\|R(U)U\|_{s}\lesssim \mathtt{C}(\|U\|_{s_0+3})\|U\|_{s}\,,\label{stimaRRR}\\
&\|R(U)U-R(V)V\|_{s}\lesssim
 \mathtt{C}(\|U\|_{s_0+3},\|V\|_{s_0+3}) \|U-V\|_{s}+\mathtt{C}(\|U\|_{s},\|V\|_{s})\|U-V\|_{s_0+3}\,,\label{nave101}\\
& \|R(U)U-R(V)V\|_{s_0+1}\lesssim
 \mathtt{C}(\|U\|_{s_0+3},\|V\|_{s_0+3}) \|U-V\|_{s_0+1}\label{magnaccio}
 \end{align}
where $\mathtt{C}(\cdot,\cdot):={C}(\max\{\cdot,\cdot\})>0$, for some non decreasing function $C(\cdot)$. Concerning the unbounded part of the equation we have for any $\s\geq 0$ and $j=1,2$
\begin{align}
\|\opbw(A_j(U;x,\xi))W\|_{\s}&\leq {C}(\|U\|_{s_0+1})\|W\|_{\s},\label{mamma1}\\
\|\opbw(A_j(U;x,\xi)-A_{j}(V;x,\xi))W\|_{\s}&\leq {C}(\|U-V\|_{s_0+1})\|W\|_{\s},\label{mamma2}
\end{align}
for any $U,V, W$ in $\cH^{\s}$.
\end{proposition}
In order to prove  Prop. \ref{NLSparapara}, we first show the following lemma.
\begin{lemma}\label{product}
Fix $s_0>d/2$ and $s\geq s_0$.
Consider $u\in H^{s}(\mathbb{T}^{d};\mathbb{C})$,
then we have that 
\begin{align}
(\pa_{\ov{u}}F)&(u,\nabla u)-\sum_{j=1}^{d}\pa_{x_j}\big(\pa_{\ov{u}_{x_{j}}}F\big)(u,\nabla u)=T_{\pa_{u\bar{u}}F}[u]+T_{\pa_{\bar{u}\,\bar{u}}F}[\bar{u}]
\label{paralin1}
\\&+\sum_{j=1}^{d}\Big(  T_{\pa_{\bar{u}u_{x_{j}}}F}[u_{x_j}]
+T_{\pa_{\bar{u}\,\ov{u}_{x_{j}}}F}[\ov{u}_{x_j}] \Big)
-\sum_{j=1}^{d}\pa_{x_j}\Big(  T_{\pa_{{u}\ov{u}_{x_{j}}}F}[u]
+T_{\pa_{\bar{u}\,\ov{u}_{x_{j}}}F}[\bar{u}] \Big)\label{paralin2}\\
&-\sum_{j=1}^{d}\pa_{x_{j}}
\sum_{k=1}^{d}\Big( 
 T_{\pa_{\ov{u}_{x_{j}} \,{u_{x_k}}}F}[u_{x_k}]
+T_{\pa_{\ov{u}_{x_{j}}\, \ov{u}_{x_{k}}}F}[\ov{u}_{x_{k}}]
\Big)+R(u)u\,,\label{paralin3}
\end{align}
where  $R(u)u$ is a remainder satisfying 
\begin{equation}\label{stimarestopara}
\|R(u)u\|_{{s}}\lesssim \mathtt{C}(\|u\|_{s_0+1})\|u\|_{{s}}^{2}\,,
\end{equation}
where $\mathtt{C}(\cdot)$ is a non decreasing function.
\end{lemma}
\begin{proof}
Since the nonlinearity is a polynomial, the \eqref{paralin1}-\eqref{paralin3} 
follow by a repeated application of  the paraproduct Lemma \ref{ParaMax}
(for more general nonlinearities one can look at \cite{Metivier2008}).
\end{proof}
\begin{proof}[proof of Prop. \ref{NLSparapara}]
Consider now the first paradifferential term in \eqref{paralin3}. 
We have, for any $j,k=1,\ldots,d$,
\[
\pa_{x_{j}}
 T_{\pa_{\ov{u}_{x_{j}} \,{u_{x_k}}}F}\pa_{x_k}u=
 \opbw(\ii\x_{j})\circ\opbw(\pa_{\ov{u}_{x_{j}} \,{u_{x_k}}}F)\circ\opbw(\ii \x_k)u\,.
\]
By applying 
Proposition \ref{compo} with $\rho=2$, 
we obtain
\begin{align}
 \opbw(\ii\x_{j})\,\circ&\,\opbw(\pa_{\ov{u}_{x_{j}} \,{u_{x_k}}}F)\circ\opbw(\ii \x_k)
=
\opbw\big( -\x_{j}\x_{k}\pa_{\ov{u}_{x_{j}} \,{u_{x_k}}}F \big)
\\&
 +\opbw\Big(
  \frac{\ii}{2 } \x_{k}\pa_{x_{j}}(\pa_{\ov{u}_{x_{j}} \,{u_{x_k}}}F)
 -\frac{\ii\x_{j}}{2}\pa_{x_{k}}(\pa_{\ov{u}_{x_{j}} \,{u_{x_k}}}F)\Big)
    \\&+R(u)u\,,
\end{align}
where$
\|{R}(u)u\|_{{s}}\lesssim \mathtt{C}(\|u\|_{{s_0+3}})\|u\|_{{s}}\,,$ for some non increasing function $\mathtt{C}(\cdot)>0$. Then \[
\begin{aligned}
-\sum_{j,k=1}^{d}\pa_{x_j} &T_{\pa_{\ov{u}_{x_{j}} \,{u_{x_k}}}F}\pa_{x_{k}}u=
\opbw\Big(\sum_{j,k=1}^{d}\x_{j}\x_{k}\pa_{\ov{u}_{x_{j}} \,{u_{x_k}}}F\Big)+
{R}(u)
\\&
-\frac{\ii}{2}\opbw\Big(
\sum_{j,k=1}^{d}\Big( -\x_{j}\pa_{x_{k}}(\pa_{\ov{u}_{x_{j}} \,{u_{x_k}}}F)
+ \x_{k}\pa_{x_{j}}(\pa_{\ov{u}_{x_{j}} \,{u_{x_k}}}F)
\Big)\Big)
\\&
\stackrel{\mathclap{\eqref{simboa2}}}{=}
\opbw(a_2(x,\x))+{R}(u)
+\frac{\ii}{2}\opbw\Big(\sum_{j,k=1}^{d} \x_{j}\pa_{x_{k}}\Big(
(\pa_{\ov{u}_{x_{j}} \,{u_{x_k}}}F)-(\pa_{\ov{u}_{x_{k}} \,{u_{x_j}}}F)\Big)
\Big)\\
&
=\opbw(a_2(x,\x))+{R}(u)\,,
\end{aligned}
\]
where we used the symmetry of the matrix $\pa_{\ov{\nabla u}\, \nabla u}F$ (recall $F$ is real).
By performing similar explicit computations on the other summands in \eqref{paralin1}-\eqref{paralin3}
we get the 
\eqref{QNLS444}, \eqref{matriceA2} with symbols in 
\eqref{simboa2}.
By the discussion above we deduced that the remainder $R(U)U$ in \eqref{QNLS444}
satisfies the bound \eqref{stimaRRR}.
The estimate \eqref{nave101} follow by the fact that the nonlinearity is polynomial and therefore all the remainders are  by multilinear.
\end{proof}
\begin{remark}\label{C1}
If the nonlinearity satisfies the hypotheses of Theorem \ref{vjj}, then the symbols of the paralinearized system are simpler, in particular $a_2$ and $b_2$ do not depend on $\xi$ and on $\nabla u$ 
\begin{equation}\label{simboa2colin}
\begin{aligned}
a_2(x)&=\, \left[h'(|u|^2)\right]^2|u|^2\,,
\quad b_2(x)=
\left[h'(|u|^2)\right]^2u^2,\\
a_1(x,\xi)&=\vec{a}_1(x)\cdot\xi=\,\left[h'(|u|^2)\right]^2
\sum_{j=1}^d\Im(u\bar{u}_{x_j})\xi_j\,. 
\end{aligned}
\end{equation}
Since they do not contain derivatives of $u$, we have the improved estimates on their semi-norms
\begin{equation*}
|a_2(x)|_{2,p,\alpha}+ |b_2(x)|_{2,p,\alpha}+ |\vec{a}_1(x)\cdot\xi|_{2,p,\alpha}\lesssim_s \mathtt{C}(\|u\|_{p}), \quad \mbox{for all}\,\,\, \tfrac{d}{2}<p.
\end{equation*}
This implies that the result of Theorem \ref{NLSparapara}, in the case that $F$ satisfies the hypotheses of Theorem \ref{vjj}, holds true changing $s_0+1\rightsquigarrow s_0$ therein. In other words, the minimal regularity needed on the solution $U$ is $\frak{s}_0>d/2+2$ instead of $d/2+3$.
\end{remark}

\section{Modified energy and linear well posedness}
In this section we shall define a norm on $H^s$ which is almost equivalent to the standard Sobolev one and which is tailored to the problem \eqref{NLSparapara},
in such a way that we shall be able to obtain \emph{a priori} estimates on the solutions. We consider first a linearized,  regularized, homogeneous  version of \eqref{QNLS444}
\begin{equation}\label{lineare}
\partial_t{U^{\epsilon}}=\ii E\opbw\big(A_{2}(\U;x,\x)+A_{1}(\U;x,\x)\big)U^{\epsilon}- \epsilon\Delta^2 U^{\epsilon},
\end{equation}
where $\Delta^2U:=\opbw(|\xi|^4)U $ and  $\U$ is a fixed function satisfying for $s_0>d/2$
\begin{equation}\label{piccolezze}
\|\U\|_{L^{\infty}H^{{s}_0+3}}+\|\partial_t{\U}\|_{L^{\infty}H^{{s}_0+1}}\leq \Theta,\quad \|{\U}\|_{L^{\infty}H^{{s}_0+1}}\leq r,
\end{equation}
for some $\Theta\geq r>0$. 
For any $\epsilon>0$ the equation \eqref{lineare} admits a unique solution defined on a small interval (depending on $\epsilon$), this is the content of the following lemma. This method is called artificial viscosity, or parabolic regularization, it was used directly on the nonlinear problem in \cite{KPV}, while it has used on the linear problem in \cite{iandolikdv, FGIM}.
\begin{lemma}\label{e-fix}
Let $\sigma\geq 0$ and assume \eqref{piccolezze}. For any $\epsilon>0$ there exists $T_{\epsilon}>0$ such that the following holds. For any initial condition $U_0=U(0,x)\in \cH^{\s}$ there exists a unique solution $U^{\epsilon}(t,x)$ of \eqref{lineare} which belongs to the space $C^0([0,T_{\epsilon});\cH^{\s})\cap C^1([0,T_{\epsilon});\cH^{\s-2})$.
\end{lemma}
\begin{proof}
We consider the operator
\begin{equation*}
\Gamma U^{\epsilon}:= e^{-\epsilon t\Delta^2}U_0+\int_0^te^{-\epsilon(t-t')\Delta^2}E\opbw(A_2+A_1)U^{\epsilon}(t')dt'.
\end{equation*}
We have $\|e^{-\epsilon t\Delta^2}U_0\|_{{\s}}\leq \|U_0\|_{H^{\s}}$ and  $\|\int_0^te^{-\epsilon(t-t')\Delta^2} f(t',\cdot)dt'\|_{{\s}}\leq t^{\frac12}\epsilon^{-\frac12}\|f\|_{{\s-2}}$, with these estimates, \eqref{semi-1}, \eqref{piccolezze} and Theorem \ref{azione} one may apply a fixed point argument in a suitable subspace of $C^{0}([0,T_{\epsilon});\cH^{\s})$ for a suitable time $T_{\epsilon}$ (going to zero when $\epsilon$ goes to zero). Let us prove the second one of the above inequalities. We use the Minkowski inequality and the boundedness of the function $\alpha^{3/2}e^{-\alpha}$ for $\alpha\geq0$, we get
\begin{align*}
\|\int_0^te^{-\epsilon(t-t')\Delta^2} f(t',\cdot)dt'\|_{{\s}}&\leq \int_0^t\|e^{-\epsilon(t-t')\Delta^2} f(t',\cdot)\|_{{\s}}dt'\\
&\sim\int_0^t\sqrt{\sum_{\xi\in \Z^d}e^{-2\epsilon(t-t')|\xi|^4}|\xi|^{2\s}|\hat{f}(t',\xi)|^2}dt'\\
&\lesssim \int_0^t\epsilon^{-\frac12}(t-t')^{-\frac12}\|f(t',\cdot)\|_{H^{\s-2}}dt'\lesssim t^{\frac12}\epsilon^{-\frac12}\|f\|_{L^{\infty}H^{\s-2}}.
\end{align*}
\end{proof}
\subsection{Diagonalization at the highest order}
Recall \eqref{matrici}, \eqref{matriceA2}, \eqref{simboa2}, consider the matrix  $E\widetilde{A}_{2}(x,\x)$
\begin{equation}\label{A2tilde}
\widetilde{A}_{2}(x,\x)
:=
\left(\begin{matrix}\widetilde{a}_{2}(x,\x) & 
\widetilde{b}_{2}(x,\x) \vspace{0.2em}\\
\ov{\widetilde{b}_{2}(x,-\x)}  &
{\widetilde{a}_{2}(x,\x)} \end{matrix}\right)\,, \quad \begin{cases}
\widetilde{a}_{2}(x,\x):=|\x|^{-2}a_2(x,\x)\,,\\
\widetilde{b}_{2}(x,\x):=|\x|^{-2}b_2(x,\x)\,.\end{cases}
\end{equation}
Define 
\begin{equation}\label{nuovadiag}
\begin{aligned}
\lambda(x,\x)&:= \sqrt{\widetilde{a}_{2}(x,\x)^{2}
-|\widetilde{b}_{2}(x,\x)|^{2}}.
\end{aligned}
\end{equation}
Notice that the symbol $\lambda$ is well-defined thanks to \eqref{elli}.
The matrix of the normalized eigenvectors associated to the eigenvalues $\pm\lambda(x,\xi)$ of 
$E\widetilde{A}_{2}(x,\x)$
is 
\begin{equation}\label{transC}
\begin{aligned}
S(x,\x)&:=\left(\begin{matrix} {s}_1(x,\x) & {s}_2(x,\x)\vspace{0.2em}\\
{\ov{s_2(x,\x)}} & {{s_1(x,\x)}}
\end{matrix}
\right)\,,
\qquad
S^{-1}(x,\x):=\left(\begin{matrix} {s}_1(x,\x) & -{s}_2(x,\x)\vspace{0.2em}\\
-{\ov{s_2(x,\x)}} & {{s_1(x,\x)}}
\end{matrix}
\right)\,,
\\
s_{1}&:=\frac{\widetilde{a}_{2}
+\lambda}{\sqrt{2\lambda\big(\widetilde{a}_{2}+
\lambda\big) }},
\qquad s_{2}:=\frac{-\widetilde{b}_{2}}{\sqrt{2\lambda\big(\widetilde{a}_{2}
+\lambda\big) }}\,.
\end{aligned}
\end{equation}
Let us also define  the lower order correction
\begin{equation}\label{simboliAA11}
\begin{aligned}
S_*(x,\x)&:=
\frac{1}{2\ii}\left( 
\begin{matrix}  \{s_{2},\overline{s_{2}}\}(x,\xi) & 2\{s_{1},{s_{2}}\}(x,\xi)\vspace{0.2em}\\
-{2\{s_{1},\overline{s_{2}}\}(x,-\xi)} & {\{s_{2},\overline{s_{2}}\}(x,-\xi)}
\end{matrix}
\right)S(x,\x)\,\\
&:=\left(
\begin{matrix}
s_{1}^{*}(x,\x) & s_{2}^{*}(x,\x) \vspace{0.2em}\\ 
\ov{s_{2}^{*}(x,-\x) } & 
{s_{1}^{*}(x,-\x)}
\end{matrix}\right).
\end{aligned}
\end{equation}
We needed to introduce such a lower order correction in order to build a parametrix (i.e. an invertible map up to smoothing operators) having a remainder gaining two derivatives, see Lemma \ref{paramatrice}.
We have the following estimates on the seminorms of the symbols in the matrices above
\begin{equation}\label{semi-S}
\begin{aligned}
&|s_1(x,\x)|_{0,p,\alpha}+|s_2(x,\x)|_{0,p,\alpha}\lesssim_s \mathtt{C}(\|\U\|_{p+1})\quad \mbox{for all}\,\,\, \tfrac{d}{2}<p,\\
&|\{m,n\}(x,\x)|_{-1,p,\alpha}\lesssim_s \mathtt{C}(\|\U\|_{p+2})\,\,\,\mbox{for all}\,\,\, m,\,n\in\{s_{1},s_2,\bar{s}_2\},\,\,\, \tfrac{d}{2}<p.
\end{aligned}
\end{equation}
\begin{lemma}\label{paramatrice}
Fix $s_0>{d}/{2}$, there exists a linear operator $R_{-2}(\U)[\cdot]$ satisfying
\begin{equation}\label{resto-paramatrice}
\|R_{-2}(\U)V\|_{s+2}\leq \mathtt{C}(\|\U\|_{{s}_0+3})\|V\|_{s},
\end{equation}
for any $s\geq {s}_0+3$ and such that
\begin{equation}
\opbw(S(x,\xi)+S_{*}(x,\xi))\opbw(S^{-1}(x,\x))V=V+R_{-2}(\U)V.
\end{equation}
\end{lemma}
\begin{proof}
It is a direct consequence of Theorem \ref{compo} used with $\rho=1$.
\end{proof}
We  diagonalize the highest order $E\opbw(A_2)$ (recall Prop. \ref{NLSparapara}) by means of the parametrix constructed above.
\begin{proposition}\label{prop:diago2}
Fix $s_0>d/2$, there exists a linear operator $R_{0}(\U)[\cdot]$,  satisfying
\begin{equation}\label{R0}
\|R_{0}(\U)V\|_{s}\leq \mathtt{C}(\|\U\|_{{s}_0+3})\|V\|_{s}\quad\mbox{for all}\,\,\, s\geq {s}_0+3,
\end{equation}
such that 
\begin{equation}\label{eq-diago-ord2}
\begin{aligned}
\opbw(S^{-1})E\opbw({A}_2+{A}_1)&\opbw(S+S_{*})V=E\opbw(\lambda(x,\xi)|\xi|^2)V\\
&+E\opbw({A}_1^{+}(x,\xi))V+R_{0}(\U)V,
\end{aligned}
\end{equation}
where $\lambda(x,\xi)$ is the eigenvalue defined in \eqref{nuovadiag} and ${A}_1^+(x,\xi)$ is a matrix of symbols having the form 
\begin{equation*}
\left(\begin{matrix}
a_1^+(x,\xi) & b_1^+(x,\xi)\\
\ov{b_1^+}(x,\xi) & {a_1^+}(x,\xi)
\end{matrix}\right),
\end{equation*}
with ${a_1^+}$ being real and $a_1^+, b_1^+$ odd with respect to $\xi$. Moreover we have the estimates on the seminorms
\begin{equation}\label{semi-D}
\begin{aligned}
|\lambda(x,\xi)|\xi|^2|_{2,p,\alpha}&\lesssim \mathtt{C}(\|\U\|_{p+1}),\quad {\mbox{for all}}\,\,\, \tfrac d2<p,\\
|{b_1^+}|_{1,p,\alpha}+|{a_1^+}|_{1,p,\alpha}&\lesssim \mathtt{C}(\|\U\|_{p+2}),\quad \mbox{for all}\,\,\, \tfrac d2<p.
\end{aligned}\end{equation}
\end{proposition}
\begin{proof}
We set the following notation: given an operator $H,$ we define $\ov{H}(u):=\ov{H(\ov{u})}$, $R_0(\U)$ is a remainder satisfying \eqref{R0} and may change from line to line throughout the proof.
We have
\begin{equation*}\label{achille21}
\begin{aligned}
&\opbw(S^{-1})
\opbw(E {A}_1)
\opbw(S)=\ii E
\left(
\begin{matrix}
C_1 & C_2\\
\ov{C_2} & \ov{C_1}
\end{matrix}
\right)
\\
&C_1:=T_{s_{1}}T_{a_{1}}T_{s_{1}}
-T_{s_{2}}T_{a_{1}}T_{\ov{s_{2}}}\,,\qquad 
C_2:=
T_{s_{1}}T_{a_{1}}T_{s_{2}}
-T_{s_{2}}T_{a_{1}}T_{{s_{1}}}\,.
\end{aligned}
\end{equation*}
By using Theorem \ref{compo} with $\rho=1$ and $s_1^2-|s_2|^2=1$, we have
$
C_1=T_{a_{1}}+R_0(\U)\,$ and $ C_2=R_0(\U).
$
 By an explicit computation, using Theorem \ref{compo}  with $\rho=2$
 and \eqref{semi-1}, \eqref{semi-S}
we have
\begin{equation}\label{achille20B}
\begin{aligned}
\opbw(S^{-1})E\opbw\Big(&{A}_2(x,\x))\Big)\opbw(S+S_*)=E\left(
\begin{matrix}
B_1 & B_2\\
\ov{B_2} & \ov{B_1}
\end{matrix}
\right)\\&+\opbw\Big(
S^{-1}E{A}_2S_*)
\Big)+R_0(\U) 
\end{aligned}
\end{equation}
where 
\begin{equation}\label{achille21B}
\begin{aligned}
B_1&:=T_{s_{1}}T_{a_{2}}T_{s_{1}}+T_{s_{1}}
T_{b_{2}}T_{\ov{s_{2}}}
+T_{s_{2}}T_{\ov{b_{2}}}T_{{s_{1}}}
+T_{s_{2}}
T_{a_{{2}}}
T_{\ov{s_{2}}}
\,,\\
B_2&:=
T_{s_{1}}T_{a_{2}}T_{s_{2}}
+T_{s_{1}}T_{b_{2}}T_{{s_{1}}}
+T_{s_{2}}T_{\ov{b_{2}}}T_{{s_{2}}}
+T_{s_{2}}T_{a_{2}}T_{{s_{1}}}\,.
\end{aligned}
\end{equation}
We study each term separately.
By using symbolic calculus, i.e. Theorem \ref{compo} with $\rho=2$, and the estimates on the seminorms \eqref{semi-1}, \eqref{semi-S},
we get \begin{equation*}
B_1:=T_{c_2}+T_{c_1}+R_0(\U)
\end{equation*}
where 
\begin{equation*}
\begin{aligned}
c_2(x,\x)&:=a_{2}(s_1^{2}+|s_2|^{2})
+b_{2} s_1\ov{s_2}+\ov{b_{2}}s_1s_2
\,,\\
c_1(x,\x)&:=\frac{1}{2\ii }\Big(\{s_{1},
a_{2}s_{1}\}
+{s_{1}}\{a_{2}, s_{1}\}
\\&+\{s_{1},b_{2}\ov{s_{2}}\}
+{s_{1}}\{b_{2},\ov{s_{2}}\}
+\{s_{2},\ov{b_{2}}s_{1}\}
\\&+
{s_{2}}\{\ov{b_{2}},s_{1}\}
+\{s_{2},a_{2}\ov{s_{2}}\}+
{s_{2}}\{a_{2}, \ov{s_{2}}\}\,\Big).
\end{aligned}
\end{equation*}
By expanding the Poisson bracket 
we get that
\begin{equation*}
c_2(x,\x)=\ov{c_2(x,\x)}\,, \qquad c_1(x,\x)=\ov{c_1(x,\x)}\,, 
\qquad
c_1(x,-\x)=-{c_1(x,\x)}\,\,.
\end{equation*}
Moreover, recalling the estimates on the seminorms \eqref{semi-1}, \eqref{semi-S},  we have
\begin{equation*}\label{achille35}
|c_1|_{1,p,\alpha }\lesssim \mathtt{C}(\|\U\|_{H^{p+2}})\quad \mbox{for all}\,\, \tfrac{d}{2}<p.
\end{equation*}
Reasoning similarly we can develop $B_2$ in \eqref{achille21B} as follows
$B_2:=T_{d_2}+T_{d_1}+R_0(\U),$
where 
\begin{equation*}\label{achille32bis}
\begin{aligned}
d_2(x,\x)&:=a_{2}s_1s_2
+b_2 s_1^{2}+\ov{b_{2}}s_2^{2}\,,\\
d_1(x,\x)&:=\frac{1}{2\ii }\Big(\{s_{1},a_{2}s_{2}\}
+{s_{1}}\{a_{2}, s_{2}\}
\\&+\{s_{1},b_{2} s_{1}\}
+{s_{1}}\{b_{2}, s_{1}\}
+\{s_{2},\ov{b_{2}}s_{2}\}
\\&+
{s_{2}}\{\ov{b_{2}},s_{2}\}
+\{s_{2},a_{2}s_{1}\}+
{s_{2}}\{a_{2}, s_{1}\}\,\Big).
\end{aligned}
\end{equation*}
By expanding the Poisson bracket 
we get that
$d_1(x,\x)\equiv0\,.$
We now study the second summand in the right hand side of \eqref{achille20B}, we compute  the matrix of symbols of order 1
\begin{equation}\label{ocazz1}
S^{-1}
E(A_2(x,\x))S_*
=
E\left(\begin{matrix}r_1(x,\x) & r_2(x,\x)
\vspace{0.2em}\\
\ov{r_2(x,-\x)} & \ov{r_1(x,-\x)}\end{matrix}
\right)\,,
\end{equation}
where
\begin{equation}\label{ocazz2}
\begin{aligned}
&r_1(x,\x):=a_{2}s_1s_1^{*}+
b_{2}s_1\ov{s_2^{*}}
+\ov{b_{2}}s_2s_1^{*}+
a_{2}s_2\ov{s_2^{*}}\,,\\
&r_2(x,\x):=
a_{2}s_1s_2^{*}+
b_{2}s_1{s_1^{*}}
+\ov{b_{2}}s_2s_2^{*}+
a_{2}s_2{s_1^{*}}\,.
\end{aligned}
\end{equation}
Moreover one can check that
 the symbols $r_1,r_2$  satisfy
\begin{equation*}
 r_1(x,\x)=\ov{r_1(x,\x)}\,, 
\qquad
r_1(x,-\x)=-{r_1(x,\x)}\,\,,\quad
r_2(x,-\x)=-{r_2(x,\x)}\,,
\end{equation*}
and, using the estimates on the seminorms \eqref{semi-1}, \eqref{semi-S}, that
\begin{equation}\label{cantolibero5}
|r_i |_{1,p,\alpha}\lesssim \mathtt{C}(\|U\|_{H^{p+2}})\,,\quad \tfrac d2<p\,,\;\;\;
i=1,2\,.
\end{equation}
One concludes by noticing that $\opbw(S^{-1}EA_2S)=E\opbw(\lambda(x,\xi)|\xi|^2)$.
\end{proof}
\begin{remark}\label{C2}
When the nonlinearity satisfies the hypotheses of Theorem \ref{vjj},  the eigenvalues \eqref{nuovadiag} and the symbols in \eqref{transC} do not depend on $\xi$. This implies, by using Theorem \ref{compo} with $\rho=2$, that $\opbw(S(x))\circ \opbw(S^{-1}(x))=\Id+R_{-2}(\U)[\cdot]$. In other words, we do not need, in this special case, corrections to obtain a parametrix with a remainder gaining two derivatives. This case is very similar, indeed, to \cite{BMM1}. As a consequence we also have that $b_1^+$ in the matrix $A_1^+$ in \eqref{eq-diago-ord2} equals to zero, indeed the off diagonal terms at order one were generated by (see \eqref{ocazz1}, \eqref{ocazz2}) $\opbw\Big(
S^{-1}E{A}_2S_*)$, which is equal to zero in this context.

\end{remark}

\subsection{Diagonalization at the sub-principal order} In this section we eliminate the off diagonal term in the matrix of symbols of order one $A^+_1(x,\xi)$ in \eqref{eq-diago-ord2}, this is possible by multiplying on the left by a matrix of paradifferential operators of order $-1$.\\
Let $\varphi(\xi)$ an even function in $C^{\infty}_c(\R)$ with $supp(\varphi)\subset\{|\xi|\geq 1/2\}$ and $\varphi\equiv 0$ on $\{|\xi|\leq 1/4\}$. We define the matrix
\begin{equation}\label{simbo-diago-1}
C(x,\xi):=\left(\begin{matrix}
0 & c(x,\xi)\\
\ov{c(x,-\xi)}& 0
\end{matrix}\right),\quad c(x,\xi):=\varphi(\xi)\frac{b^{+}_1(x,\xi)}{\lambda(x,\xi)|\xi|^2},
\end{equation}
where $b_1^+(x,\xi)$ is given by Prop. \ref{prop:diago2} and $\lambda(x,\xi)$ is defined in \eqref{nuovadiag}.
We note that, thanks to the ellipticity \eqref{elli} and the estimates \eqref{semi-D}, $c(x,\xi)$ is a well defined symbol of order $-1$ verifying
\begin{equation}\label{semi-D1}
|c|_{-1,p,\alpha}\lesssim \mathtt{C}(\|\U\|_{p+2}), \quad \mbox{for all}\,\,\, \tfrac d2<p.
\end{equation}
\begin{lemma}\label{diago1}
There exists an operator $R_0(\U)[\cdot]$, different from the one in Prop. \ref{prop:diago2}, satisfying the estimate \eqref{R0}, such that 
\begin{equation*}
\begin{aligned}
\big(\Id-\opbw(C(x,\xi))\big)&E\big(\opbw(\lambda(x,\xi)|\xi|^2)+\opbw(A_1^+(x,\xi))\big)=\\
&E \opbw(\lambda(x,\xi)|\xi|^2)+\Id \opbw(a_1^+(x,\xi))+R_0(\U),
\end{aligned}
\end{equation*}
where $a_1^+(x,\xi)$ is the one of Prop. \ref{prop:diago2} and $\lambda(x,\xi)$ is defined in \eqref{nuovadiag}.
\end{lemma}
\begin{proof}
We begin by noticing that $\opbw(C(x,\xi))E\opbw(A_1^+(x,\xi))$ may be absorbed in the remainder $R_0(\U)$ thanks to Theorem \ref{azione}, \eqref{semi-D} and \eqref{semi-D1}. We use symbol calculus, i. e. Theorem \ref{compo} with $\rho=1$ and we obtain
\begin{equation*}
\opbw(C(x,\xi))E\opbw(\lambda(x,\xi)|\xi|^2)=\opbw\left(\begin{matrix}   0 & |\xi|^2\lambda(x,\xi) c(x,\xi)\\
               											|\xi|^2\lambda(x,\xi)\ov{c}(x,-\xi)&0\end{matrix} \right) +R_0(\U),
\end{equation*}
from which, thanks to the choice of $c(x,\xi)$ in \eqref{simbo-diago-1}, we conclude the proof.
\end{proof}
\begin{remark}\label{C3}
This step is not needed in the case that the nonlinearity satisfies the hypothesis of Theorem \ref{vjj}. 
\end{remark}
\subsection{The modified energy} We define the operator
\begin{equation}\label{C-op}
\cC:=(\Id-\opbw(C(x,\xi)))\opbw(S^{-1}(x,\xi)),
\end{equation}
where $C(x,\xi)$ is defined in \eqref{simbo-diago-1} and $S^{-1}(x,\xi)$ is defined in \eqref{transC}.
We introduce the following norm 
\begin{equation}\label{modifica}
\|U\|_{\U,s}^2:=\langle\opbw(\lambda^s(x,\xi)|\xi|^{2s})\cC U, \cC U \rangle_{L^2}.
\end{equation}
We  prove that this norm is almost equivalent to the standard Sobolev one $\|\cdot\|^2_{\s}$.
\begin{lemma}\label{equivalenza}
Fix $\sigma\geq0$ and $r>0$ as in \eqref{piccolezze}. There exists a constant $C_r>0$ such that 
\begin{equation*}
C_r^{-1}\|V\|_{\sigma}^2-\|V\|_{-2}^2\leq \|V\|_{\U,\s}^2\leq C_r\|V\|_{\s}^2, \quad \mbox{for all}\,\,\, V\in \cH^{\s}.
\end{equation*}
\end{lemma}
\begin{proof}
The upper bound follows by Theorem \ref{azione}, \eqref{semi-D} and \eqref{semi-D1}, indeed
\begin{equation*}
\|V\|_{\U,\s}^2\leq\|\opbw(\lambda^{\s}(x,\xi)|\xi|^{2\s})\cC V\|_{-\s}\|\cC V\|_{\s}\leq C_r \|V\|_{\s}^2.
\end{equation*}
We focus on the lower bound. Let $\delta>0$ such that $s_0-\delta>d/2$, we use Theorem \ref{compo} with $\rho=\delta$ and we obtain
\begin{equation}\label{comunione1}
\opbw(\lambda^{-\frac{\s}{2}})\opbw(S)\opbw(\lambda^{\frac{\s}{2}})\cC=\Id+R_{-\delta}(\U),
\end{equation}
where $\|R_{-\delta}(\U)V\|_{\sigma}\lesssim C_{\Theta}\|V\|_{\sigma-\delta}$. Analogously we obtain
\begin{equation}\label{comunione2}
\opbw(\lambda^{\frac{\s}{2}})\opbw(|\xi|^{2\s})
\opbw(\lambda^{\frac{\s}{2}})=\opbw(\lambda^{\s}|\xi|^{2\s})+R_{2\s-\delta}(\U),
\end{equation}
where $\|R_{2\s-\delta}(\U)\|_{\s'-2\s+\delta}\leq C_{\Theta}\|U\|_{\s'}$. We have the following chain of inequalities
\begin{equation*}
\begin{aligned}
\|V\|^2_{\s}&\stackrel{\mathclap{\eqref{comunione1}}}{\leq} \|\opbw(\lambda^{-\frac{\s}{2}})\opbw(S)\opbw(\lambda^{\frac{\s}{2}})\cC V\|_{\s}^2+C_{\Theta}\|V\|_{\s-\delta}^2\\
&\leq C_r\|\opbw(\lambda^{\frac{\s}{2}})\cC V\|_{\s}^2+C_{\Theta}\|V\|_{\s-\delta}^2\\
&=C_{r}\langle \opbw(\lambda^{\frac{\s}{2}})\opbw(|\xi|^{2\s})\opbw(\lambda^{\frac{\s}{2}})\cC V,\cC V\rangle+C_{\Theta}\|V\|_{\s-\delta}^2\\
&\stackrel{\mathclap{\eqref{comunione2}}}{\leq} C_r\|V\|_{\U,\s}^2+C_{\Theta}\|V\|_{\s-{\delta}/{2}}^2.
\end{aligned}\end{equation*}
By using the interpolation inequality \eqref{interpolo}, Young inequality $ab\leq p^{-1}a^p+q^{-1}b^q$ with $p=\frac{2(\s+2)}{\delta}$, $q=\frac{2(\s+2)}{2(\s+2)-\delta}$, we obtain $\|V\|_{\s-\frac{\delta}{2}}^2\leq \eta^{-\frac{2(\s+2)}{\delta}}\|V\|_{-2}^2+\eta^{\frac{2(\s+2)}{2(\s+2)-\delta}}\|V\|_{\s}^2$ for any $\eta>0$. We conclude by choosing $\eta$ small enough.
\end{proof}

\begin{lemma}[Garding type inequality]\label{garding} Let $\U$ be  as in \eqref{piccolezze}, there exist $C_{\Theta}, C_r>0$ such that
\begin{align*}
&\langle\opbw(\lambda^{\s}|\xi|^{2\s})\cC \Delta^2U,\cC U\rangle\geq C_r\|U\|_{\s+2}^2-C_{\Theta}\|U\|_{\s}^2,\\
&\langle \opbw(\lambda^{\s}|\xi|^{2\s})\cC U,\cC \Delta^2U\rangle\geq C_r\|U\|_{\s+2}^2-C_{\Theta}\|U\|_{\s}^2.
\end{align*}
\end{lemma}
\begin{proof}
We prove the first inequality, being the second one  similar.
By  using symbolic calculus, i.e. Theorem \ref{compo}, with $\rho=1$ we obtain an operator $R_{2\sigma+3}(\U)$ satisfying $\|R_{2\sigma+3}(\U)V\|_{0}\leq C_{\Theta}\|V\|_{2\s+3}$ such that
\begin{equation*}
\opbw(\lambda^{\s}|\xi|^{2\s})\cC\opbw(|\xi|^4)-\opbw(|\xi|^2)\opbw(\lambda^{\s}|\xi|^{2\s})\cC\opbw(|\xi|^2)=R_{2\sigma+3}(\U).
\end{equation*}
As a consequence of the above equation we have
\begin{equation}\label{G0}
\begin{aligned}
&\langle\opbw(\lambda^{\s}|\xi|^{2\s})\cC \Delta^2U,\cC U\rangle\\ &=\langle\opbw(|\xi|^2)\opbw(\lambda^{\s}|\xi|^{2\s})\cC\opbw(|\xi|^2)U,\cC U\rangle+\langle R_{2\s+3}(\U)U,\cC U\rangle\\
&=\|\opbw(|\xi|^2)U\|_{\U,\sigma}+\langle\opbw(\lambda^{\s}|\xi|^{2\s})\cC\opbw(|\xi|^2)U, R_{1}(\U) U\rangle\\
&\quad+\langle R_{2\s+3}(\U)U,\cC U\rangle,
\end{aligned}\end{equation}
where in the last equality we have used the self adjoint character of $\opbw(|\xi|^2)$ and  $\opbw(|\xi|^2)\cC-\cC\opbw(|\xi|^2)=R_{1}(\U)$, with $\|R_{1}(\U)\|_{\s}\leq C_{\Theta}\|U\|_{\s+1}$ for any $\s\geq 0$. By means of the duality inequality $\langle\cdot,\cdot\rangle\leq \|\cdot\|_{-\s-1/2}\|\cdot\|_{\s+1/2}$, the action Theorem \ref{azione}, the estimates on the semi norms \eqref{semi-D}, \eqref{semi-D1}, we infer that
\begin{equation}\label{G1}
\langle\opbw(\lambda^{\s}|\xi|^{2\s})\cC\opbw(|\xi|^2)U, R_{1}(\U) U\rangle\leq C_{\Theta}\|U\|_{\sigma+3/2}^2.
\end{equation}
Analogously, by using $\langle\cdot,\cdot\rangle\leq \|\cdot\|_{-\s-3/2}\|\cdot\|_{\s+3/2}$ we obtain 
\begin{equation}\label{G2}
\langle R_{2\s+3}(\U)U,\cC U\rangle\leq C_{\Theta}\|U\|^2_{\s+3/2}.
\end{equation}
Recalling \eqref{G0}, we use the left inequality in Lemma \ref{equivalenza} and inequalities \eqref{G1}, \eqref{G2} to infer that
\begin{equation*}
\langle\opbw(\lambda^{\s}|\xi|^{2\s})\cC \Delta^2U,\cC U\rangle\geq C_r^{-1}\|U\|_{\s+2}^2-\|U\|_{\s}^2-C_{\Theta}\|U\|^2_{\s+3/2},
\end{equation*}
from which one obtains the thesis by using the  interpolation inequality
$\|U\|^2_{\s+3/2}\leq \eta \|U\|^2_{\s+2}+\eta^{-1}\|U\|^2_{\s}$ with $\eta$ small enough.
\end{proof}

In the following we prove  the \emph{energy estimates} on the solutions of the linear problem \eqref{lineare}.
\begin{proposition}\label{energy} Let $\U$ satisfy \eqref{piccolezze}. For any $\sigma>0$ there exist constants $C_r>0$ and $C_{\Theta}>0$ such that the unique solution of \eqref{lineare} fullfills
\begin{equation*}
\|U^{\epsilon}\|_{\sigma}^2\leq C_r\|U_0\|_{\sigma}^2+C_{\Theta}\int_0^t\|U^{\epsilon}(\tau)\|_{\s}^2d\tau, \quad \mbox{for all} \,\,\,t\in[0,T).
\end{equation*}
As a consequence one also has
\begin{equation}\label{gro}
\|U^{\epsilon}\|_{\sigma}\leq C_re^{C_{\Theta}t}\|U_0\|_{\sigma}, \quad \mbox{for all} \,\,\, t\in[0,T).
\end{equation}
\end{proposition}
\begin{proof}
We take the time derivative of the energy \eqref{modifica} along the solutions of the equation \eqref{lineare}. We obtain
\begin{equation}\label{primaderivata}\begin{aligned}
\tfrac{d}{dt}\|U^{\epsilon}\|_{\U,\s}^2= &\langle \opbw(\tfrac{d}{dt}(\lambda(x,\xi)^{\s})|\xi|^{2\s}\cC U^{\epsilon},\cC U^{\epsilon}\rangle\\
&+\langle \opbw(\lambda(x,\xi)^{\s}|\xi|^{2\s})\cC_t U^{\epsilon},\cC U^{\epsilon}\rangle+\langle \opbw(\lambda(x,\xi)^{\s}|\xi|^{2\s})\cC U^{\epsilon},\cC_t U^{\epsilon}\rangle\\
&+\langle \opbw(\lambda(x,\xi)^{\s}|\xi|^{2\s})\cC U^{\epsilon}_t,\cC U^{\epsilon}\rangle+\langle \opbw(\lambda(x,\xi)^{\s}|\xi|^{2\s})\cC U^{\epsilon},\cC U^{\epsilon}_t\rangle.
\end{aligned}\end{equation}
The first summand in the r.h.s. of \eqref{primaderivata} is bounded from above by $C_{\Theta}\|U^{\epsilon}\|_\sigma^2$, indeed, by using   \eqref{piccolezze}
\begin{equation*}
|\partial_t\lambda(x,\xi)|_{0,s_0,\alpha}\lesssim \mathtt{C}(\|\partial_{t}\U\|_{s_0})\leq C_{\Theta},
\end{equation*}
therefore, thanks also to \eqref{semi-S} and \eqref{semi-D1} one is in position to use Theorem \ref{azione} together with the duality inequality $\langle\cdot,\cdot\rangle\leq\|\cdot\|_{-\s}\|\cdot\|_{\s} $. With an analogous reasoning one may bound from above the second line in \eqref{primaderivata} by $C_{\Theta}\|U^{\epsilon}\|_\sigma^2$.\\
We focus on the third line of \eqref{primaderivata}, by using the equation \eqref{lineare} and setting $A:=A_2(\U;x,\xi)+A_1(\U;x,\xi)$, we get
\begin{align}
\langle& \opbw(\lambda^{\s}|\xi|^{2\s})\cC U^{\epsilon}_t,\cC U^{\epsilon}\rangle+\langle \opbw(\lambda^{\s}|\xi|^{2\s})\cC U^{\epsilon},\cC U^{\epsilon}_t\rangle\nonumber\\
=&\langle\opbw(\lambda^{\s}|\xi|^{2\s})\cC\ii E \opbw(A)U^{\epsilon},\cC U^{\epsilon}\rangle+\langle\opbw(\lambda^{\s}|\xi|^{2\s})\cC U^{\epsilon},\cC \ii E \opbw(A)U^{\epsilon}\rangle\label{freschi}\\
&- \langle\opbw(\lambda^{\s}|\xi|^{2\s})\cC \epsilon\Delta^2U^{\epsilon},\cC U^{\epsilon}\rangle-\langle \opbw(\lambda^{\s}|\xi|^{2\s})\cC U^{\epsilon},\cC \epsilon\Delta^2U^{\epsilon}\rangle\label{calorosi}.
\end{align}
Concerning \eqref{calorosi} we may use Lemma \ref{garding} and bound it from above by
\begin{equation*}
-C_r\|U^{\epsilon}\|_{\s+2}^2+C_{\Theta}\|U^{\epsilon}\|_{\s}^2\leq C_{\Theta}\|U^{\epsilon}\|_{\s}^2.
\end{equation*}
In \eqref{freschi} we have to see a cancellation. By Lemma \ref{paramatrice} we have
\begin{equation*}
\begin{aligned}
\eqref{freschi}&= \langle\ii\opbw(\lambda^{\s}|\xi|^{2\s})\cC E \opbw(A)\opbw(S+S_{*})\opbw(S^{-1})U^{\epsilon},\cC U^{\epsilon}\rangle\\
&+\langle\opbw(\lambda^{\s}|\xi|^{2\s})\cC U^{\epsilon},\ii\cC  E \opbw(A)\opbw(S+S_{*})\opbw(S^{-1})U^{\epsilon}\rangle
\end{aligned}\end{equation*}
modulo a remainder which, by means of \eqref{resto-paramatrice}, Cauchy-Schwarz inequality and Theorem \ref{azione}, is bounded from above by $C_{\Theta}\|U^{\epsilon}\|_{\sigma}^2$. Recalling \eqref{C-op}, we are in position to use Prop. \ref{prop:diago2} and Lemma \ref{diago1} and obtain
\begin{align*}
\eqref{freschi}&=\langle \ii  \opbw(\lambda^{\s}|\xi|^{2\s}) \opbw(E\lambda|\xi|^{2}+\Id a_1^+\big)\opbw(S^{-1})U^{\epsilon}, \cC U^{\epsilon}\rangle\\
&+\langle \opbw(\lambda^{\s}|\xi|^{2\s})\cC U^{\epsilon}, \ii  \opbw(E\lambda|\xi|^{2}+\Id a_1^+\big)\opbw(S^{-1})U^{\epsilon}\rangle.
\end{align*}
modulo remainders which are bounded from above by $C_{\Theta}\|U^{\epsilon}\|_{\sigma}^2$. 
Recalling \eqref{C-op} and using the skew-adjoint character of the operator $ \ii  \big(\opbw(E\lambda|\xi|^{2}+\Id a_1^+\big)$, we may write 
\begin{align*}
&\eqref{freschi}=\Big\langle \ii  \Big[\opbw(\lambda^{\s}|\xi|^{2\s}), \cA\Big]\opbw(S^{-1})U^{\epsilon}, \opbw(S^{-1})U^{\epsilon}\Big\rangle-\\
&\Big\langle \ii  \Big[\opbw(\lambda^{\s}|\xi|^{2\s}), \cA\Big]\opbw(C)\opbw(S^{-1})U^{\epsilon}, \opbw(C)\opbw(S^{-1})U^{\epsilon}\Big\rangle,
\end{align*}
where we have defined $\cA:=\opbw(E\lambda|\xi|^{2}+\Id a_1^+)$.
We discuss the first line in the above equation, the second one is similar and lower order. By using the symbolic calculus, i.e. Theorem \ref{compo}, with $\rho=2$ we obtain 
\begin{equation*}
\Big\langle\Big[\Id\opbw(\lambda^{\s}|\xi|^{2\s}), \cA\Big]V,V\Big\rangle= \Big\langle\opbw\big(E\big\{\lambda^{\s}|\xi|^{2\s},\lambda|\xi|^{2}\big\}\big)V,V\Big\rangle,
\end{equation*}
where $V:=\opbw(S^{-1})U^{\epsilon}$ and modulo remainders bounded by $C_{\Theta}\|U^{\epsilon}\|_{\s}^2$. We conclude the proof noticing that the above Poisson bracket equals to zero.\\
We eventually obtained $\tfrac{d}{dt}\|U^{\epsilon}\|_{\U,\s}^2\leq C_{\Theta}\|U^{\epsilon}\|^2_{{\sigma}}$, integrating over the time interval $[0,t)$ we obtain
\begin{equation*}
\|U^{\epsilon}\|_{\U(t),\s}^2\leq \|U^{\epsilon}(0)\|_{\U(0),\s}^2+C_{\Theta}\int_0^{t}\|U^{\epsilon}(\tau)\|_{{\sigma}}^2d\tau\leq C_r \|U^{\epsilon}(0)\|_{{\sigma}}^2+C_{\Theta}\int_0^{t}\|U^{\epsilon}(\tau)\|_{{\sigma}}^2d\tau.
\end{equation*}
We now use \eqref{equivalenza} and the fact that $\|\partial_{t}U^{\epsilon}\|_{{-2}}\leq C_{\Theta}\|U^{\epsilon}\|_{{0}}\leq C_{\Theta} \|U^{\epsilon}\|_{{\sigma}}$ since $\sigma\geq 0.$
\end{proof}
We are in position to state and prove the following linear well posedness result.
\begin{proposition}\label{ex-lineare}
Let $\Theta\geq r>0$ and $\U$ be a function in $C^0([0,T); \cH^{s_0+3})\cap C^1([0,T); \cH^{s_0+1})$, satisfying \eqref{piccolezze} with $s_0>d/2$. Let $\sigma\geq 0$ and $R(t)$  be a function in $C^0([0,T);\cH^{\s})$. Then there exists a unique solution 
$V$ in $C^0([0,T); \cH^{\s})\cap C^1([0,T); \cH^{\s-2})$ of the linear inhomogeneous problem
\begin{equation}\label{lineare-tot}
\partial_t V=\ii E\opbw\big(A_{2}(\U;x,\x)+A_{1}(\U;x,\x)\big)V+
R(t)\,,
\end{equation}
with initial condition $V(0,x)=V_0(x)\in \cH^{\s}$. Furthermore the solution satisfies 
\begin{equation}\label{stima-lineare}
\|V\|_{L^{\infty}H^{\s}}\leq C_{r,\s} e^{C_{\Theta,\s}T}\|V_0\|_{\s}+TC_{\Theta,\s}e^{C_{\Theta,\s}T}\|R(t)\|_{L^{\infty}H^{\s}}.
\end{equation}
\end{proposition}
\begin{proof}
We consider the following smoothed version of the initial condition $V_0$
\begin{equation*}
V_0^{\epsilon}:=\chi(\epsilon^{\frac18}|D|)V_0
:=\mathcal{F}^{-1}(\chi(\epsilon^{\frac18}|\xi|)\hat{V}_0(\xi))\,,
\end{equation*}
where $\chi$ is a $C^{\infty}$ function with compact support being equal to one on $[-1,1]$ 
and zero on $\R\setminus [-2,2]$. 
We consider moreover the smoothed, 
homogeneous version of \eqref{lineare-tot}, i.e. equation \eqref{lineare}. 
All the solutions $V^{\epsilon}$ are defined 
on a common time interval $[0,T)$ with $T>0$ independent on $\epsilon$, because of Proposition \ref{energy}.
We prove that the sequence $V^{\epsilon}$ 
converges to a solution of \eqref{lineare} with 
$\epsilon$ equal to zero both in the initial condition and in the equation.
Let  $0<\epsilon'<\epsilon$, set $V^{\epsilon,\epsilon'}=V^{\epsilon}-V^{\epsilon'}$, then 
\begin{equation}\label{aux}
\partial_tV^{\epsilon,\epsilon'}=\ii E\opbw(A_2(\U;x,\xi)+A_1(\U;x,\xi))V^{\epsilon,\epsilon'}-\epsilon{\Delta^2}V^{\epsilon,\epsilon'}
+{\Delta^2}V^{\epsilon}(\epsilon-\epsilon')\,.
\end{equation}
Thanks to the discussion above there exists the flow $\Phi(t)$ of the equation 
\[
\partial_tV^{\epsilon,\epsilon'}=\ii E\opbw(A_2(\U;x,\xi)+A_1(\U;x,\xi))V^{\epsilon,\epsilon'}-\epsilon{\Delta^2}V^{\epsilon,\epsilon'}\,,
\] 
and it has estimates independent of $\epsilon,\epsilon'$. 
By means of Duhamel formula, we can write the solution of \eqref{aux} in the implicit form
\begin{equation*}
V^{\epsilon,\epsilon'}(t,x)=\Phi(t)(V_0^{\epsilon'}-V_0^{\epsilon})+(\epsilon'-\epsilon)\Phi(t)\int_0^t\Phi(s)^{-1}{\Delta^2}V^{\epsilon}(s,x)ds\,,
\end{equation*}
using  \eqref{gro} we obtain 
$\|V^{\epsilon,\epsilon'}(t,x)\|_{L^{\infty}H^{\sigma}}\leq 
C (\epsilon-\epsilon')\|V_0\|_{ H^{\sigma}}+(\epsilon-\epsilon')\|V_0^{\epsilon}(t)\|_{H^{\sigma+4}}$ . 
We conclude using {the smoothing estimate} $\|V_0^{\epsilon}\|_{H^{\sigma+4}}\leq \epsilon^{-\frac12}\|V_0\|_{H^{\sigma}}$.  
We eventually proved that we have a well defined 
flow of the equation \eqref{lineare-tot} with $\mathfrak{R}(t)=0$. 
The non homogeneous case, i.e. $\mathfrak{R}(t)\neq 0$, 
follows  by using the Duhamel formula.
\end{proof}
\begin{remark}\label{C4}
In the case that the nonlinearity satisfies the hypotheses of Theorem \ref{vjj}, the modified energy in \eqref{C-op}, \eqref{modifica} is defined by the simpler operator 
$\cC:=\opbw(S^{-1}(x,\xi)).$
This  modified energy is similar to the one in \cite{BMM1}. As consequence of this and of Remarks \ref{C1}, \ref{C2}, \ref{C3} the \eqref{piccolezze} may be replaced by
\begin{equation}\label{Cpic}
\|\U\|_{L^{\infty}H^{{s}_0+2}}+\|\partial_t{\U}\|_{L^{\infty}H^{{s}_0}}\leq \Theta,\quad \|{\U}\|_{L^{\infty}H^{{s}_0}}\leq r,
\end{equation}
so that Prop. \ref{ex-lineare} holds for $\U$ satisfying \eqref{Cpic}.
\end{remark}

\section{Nonlinear well posedness} This section is nowadays standard, we mostly follow \cite{BMM1}, analogous reasonings are developed  also in \cite{FIJMPA,FGIM}. 
We set $\cA(\U):=\ii E\opbw(A_2(\U;x,\xi)+A_{1}(\U;x,\xi))$, $U^0(t,x)=U_0(x)$ and we define the following iterative scheme (recall \eqref{QNLS444})
\begin{equation}\label{katozzo}
\cP_n:=\begin{cases}
\partial_t U^n=\cA(U^{n-1})U^n+R(U^{n-1})U^{n-1} \\
U^n(0,x)=U_0(x)
\end{cases},\quad \mbox{for all} \quad n\geq 1.
\end{equation}
In the following lemma we prove that each problem $\cP_n$, admits a unique solution for any $n\geq 1$, and that the sequence of such solutions converges.
\begin{lemma}\label{iterativo}
Let $U_0\in\cH^{s}$, $s>d/2+3$ and set $r=2\|U_0\|_{H^{s_0+1}}$, $s_0>d/2$. There exists $T:=T(s,\|U_0\|_{s_0+3})$ such that for any $n\geq 1$ we have the following.\\
$(S1)_n$- The problem $\cP_n$ admits a unique solution $U^n$ belonging to the functional space $C^0([0,T);\cH^{s})\cap C^1([0,T);\cH^{s-2})$.\\
$(S2)_n$- There exists a constant $C_r>0$ such that, setting $\Theta:= 4C_r\|U_0\|_{s_0+3}$ and $M:= 4C_r\|U_0\|_{s}$  we have for any $1\leq m\leq n$
\begin{align}
&\|U^m\|_{L^{\infty}H^{s_0+1}}\leq r,\label{1}\\
&\|U^m\|_{L^{\infty}H^{s_0+3}}\leq \Theta, \quad \|\partial_tU^m\|_{L^{\infty}H^{s_0+1}}\leq C_r \Theta,\label{2}\\
&\|U^m\|_{L^{\infty}H^{s}}\leq M, \quad\quad \|\partial_tU^m\|_{L^{\infty}H^{s-2}}\leq C_r M\label{3}.
\end{align}
$(S3)_n$- For any $1\leq m\leq n$ we have $\|U^m-U^{m-1}\|_{L^{\infty}H^{s_0+1}}\leq 2^{-m}r$.
\end{lemma}
\begin{proof}
We proceed by induction over $n\geq 1.$ We prove $(S1)_1$ by using  Prop. \ref{ex-lineare} with $R(t):=R(U_0)U_0$. Moreover  by using \eqref{stima-lineare} we have
\begin{equation*}
\|U^1\|_{L^{\infty}{ H^{\s}}}\leq C_{r,\s}e^{C_{\Theta}T}\|U_0\|_{\s}+C_{\Theta}e^{TC_{\Theta}}\|U_0\|_{\s},
\end{equation*}
for any $\sigma\geq s_0+1$. By using the latter inequality with $\s= s_0+3$, the choice of $\Theta$ in the statement and $T>0$ such that $TC_{\Theta} e^{TC_{\Theta}}\leq 1/4$ we obtain the first in \eqref{2} at the level $n=1$. To obtain the second in \eqref{2} one has to use directly the equation \eqref{katozzo} together with the estimates \eqref{stimaRRR}, \eqref{mamma1} and the first in \eqref{2}
The \eqref{3} at the level $n=1$ is obtained similarly. To prove \eqref{1} at the level $n=1$ we note that $V^1:=U^1-U_0$ solves the equation
\begin{equation*}
\partial_t V_1=\cA(U_0)V^1+R(U_0)U_0+\cA(U_0)U_0,
\end{equation*}
with initial condition $V^1(0,x)=0$. Using \eqref{stima-lineare}, \eqref{mamma1} and \eqref{stimaRRR}, we may bound $\|V^1\|_{L^{\infty}H^{\s}}$by $C_rC_{\Theta}\Theta Te^{C_{\Theta}T}$, therefore we prove \eqref{1} at the level $n=1$ and $(S3)_1$ by choosing $T$ in  such a way that $C_rC_{\Theta}\Theta Te^{C_{\Theta}T}\leq r/2$.\\
We now assume that $(S1)_{n-1},$  $(S2)_{n-1},$ $(S3)_{n-1},$ hold true and we prove them at the level $n$. Owing to \eqref{1},\eqref{2}, \eqref{3} at the level $n-1$ we can use Prop. \ref{ex-lineare} in order to show $(S1)_n$. We prove \eqref{2}. By using \eqref{stima-lineare} with $\sigma=s_0+3$, we obtain
\[
\begin{aligned}
\|U^n\|_{L^{\infty}H^{\s}}&\leq C_{r,\s} e^{C_{\Theta}T}\|U_0\|_{\s}+TC_{\Theta,\s}e^{C_{\Theta}T}\|R(U^{n-1})U^{n-1}\|_{L^{\infty}H^{\s}}\\
&\leq e^{C_{\Theta}T}\tfrac{\Theta}{4}+TC_{\Theta} e^{TC_{\Theta}},
\end{aligned}
\]
therefore we obtain the first in \eqref{2} by choosing $T\Theta$ small enough. To obtain the second in \eqref{2}, we can reason as done in the case of \eqref{2} at the level $n=1$. The \eqref{3} is obtained in a similar way, by using Prop. \ref{ex-lineare} with $\s=s$. The \eqref{1} is a consequence of $(S3)_n$. We prove $(S3)_n$. Set $V^n:=U^{n}-U^{n-1}$, we have
\begin{equation*}
\begin{aligned}
&\begin{cases}
\partial_t V^{n}=\cA(U^{n-1})V^n+f_n, \\
V^n(0,x)=0
\end{cases}
\end{aligned}\end{equation*}
where
\begin{equation*}
\begin{aligned}
 f_n:&=\Big[\cA(U^{n-1})-\cA(U^{n-2})\Big]U^{n-1}+R(U^{n-1})U^{n-1}-R(U^{n-2})U^{n-2}.
\end{aligned}\end{equation*}
We note that, by means of \eqref{mamma2} and \eqref{magnaccio} we have
\begin{equation*}
\|f_n\|_{s_0+1}\lesssim C_{\Theta}\|V^{n-1}\|_{s_0+1}.
\end{equation*}
We are in position to use Prop. \ref{ex-lineare} with $\s=s_0+1$ and $R(t):= f_{n}$  
\begin{equation*}
\|V^n\|_{L^{\infty}H^{s_0+1}}\leq TC_{\Theta,\s}e^{C_{\Theta}T}\|f_n\|_{L^{\infty}H^{s_0+1}}\leq e^{C_{\Theta}T}TC_{\Theta}\|V^{n-1}\|_{s_0+1},
\end{equation*}
we conclude again by choosing $T$ small enough.
\end{proof}
\begin{remark}\label{C5}
In view of Remarks \ref{C1}, \ref{C2}, \ref{C3}, \ref{C4} if the nonlinearity satisfies the hypotheses of Theorem \ref{main}, we may replace $s_0+1\rightsquigarrow s_0$ in the statement of Lemma \ref{iterativo}.
\end{remark}
We are now in position to prove the main theorem.
\begin{proof}[proof of Theorems \ref{main} and \ref{vjj}]
We prove Theorem \ref{main}. By means of Prop. \ref{NLSparapara} we know that \eqref{NLS} is equivalent to \eqref{QNLS444}.
As a consequence of Lemma \ref{iterativo} we obtain a Cauchy sequence in $U^n\in C^0([0,T),\cH^{s'})\cap C^1([0,T),\cH^{s'-2})$ for $s>s'\geq s_0+1$. Indeed, for $s'=s_0+1$ this is the $(S3)_n$ and for $s>s'> s_0+1$ we can interpolate $(S3)_n$ and \eqref{3} by means of \eqref{interpolo}. Analogously one proves that $\partial_t U^n$ is a Cauchy sequence in $C^0([0,T),\cH^{s'-2})$. Let $U(t)$ be the limit, in order to show that $U(t)$ solves \eqref{QNLS444} it is enough to prove that 
\[
\cA(U^{n})U^{n}-\cA(U)U+R(U^{n})U^{n}-R(U)U,
\]
converges to $0$ in $L^{\infty}\cH^{s'-2}$, but this is a consequence of Theorem \ref{azione} and contraction estimates \eqref{mamma2}, \eqref{nave101}. The uniqueness may be proved by contradiction with similar computations to the ones performed in Lemma \ref{iterativo}.  Thanks to $\eqref{3}$ we have that $U^n$ is a bounded sequence in $C^0([0,T),\cH^{s})\cap C^1([0,T),\cH^{s-2})$, which implies that $U\in L^{\infty}([0,T),\cH^{s})\cap Lip([0,T),\cH^{s-2})$. To prove that it is actually continuous with the $\cH^s$ topology, as well as the continuity of the solution map, one has to use the Bona-Smith technique \cite{BSkdv} as done in \cite{BMM1} or \cite{FIJMPA}. We do not reproduce the proof here.\\
Concerning the proof of Theorem \ref{vjj} one has to reason exactly in the same way taking into account Remarks \ref{C1}, \ref{C2}, \ref{C3}, \ref{C4}, \ref{C5}.
\end{proof}

\gr{Acknowledgements.} 

F. Iandoli has been partially supported 
by research project PRIN 2017JPCAPN: ``Qualitative and quantitative aspects of nonlinear PDEs" of the 
Italian Ministry of Education and Research (MIUR).

\bibliographystyle{plain}

\end{document}